\documentclass[reqno,draft]{amsart}
\usepackage{amsmath,amssymb,amsfonts}
\usepackage{euscript}
\usepackage{enumerate}
\newtheorem{theorem}{Theorem}[section]
\newtheorem{lemma}[theorem]{Lemma}
\newtheorem{example}[theorem]{Example}

\renewcommand{\epsilon}{\varepsilon}
\newcommand{\eps}{\varepsilon}
\newcommand{\lbd}{\lambda}
\newcommand{\dint}{\displaystyle\int}
\newcommand{\dr}{\, dr}
\DeclareMathOperator{\Id}{Id}
\DeclareMathOperator{\e}{e}
\newcommand{\prts}[1]{\left(#1\right)}
\newcommand{\prtsr}[1]{\left[#1\right]}
\newcommand{\pfrac}[2]{\prts{\dfrac{#1}{#2}}}
\newcommand{\abs}[1]{\left|#1\right|}
\newcommand{\set}[1]{\left\{#1\right\}}
\newcommand{\setm}[1]{\setminus\set{#1}}
\newcommand{\norm}[1]{\left\|#1\right\|}
\usepackage{dsfont}
   \newcommand{\N}{\ensuremath{\mathds N}}
   \newcommand{\R}{\ensuremath{\mathds R}}

\def\cB{\EuScript{B}}

\def\cV{\EuScript{V}}
\def\cX{\EuScript{X}}
\def\cP{\EuScript{P}}
\def\cW{\EuScript{W}}
\begin{document}
\title[Generalized nonuniform dichotomies and local stable manifolds]
   {Generalized nonuniform dichotomies and local stable manifolds}
\author[Ant\'onio J. G. Bento]
   {Ant\'onio J. G. Bento}
\address{Ant\'onio J. G. Bento\\
   Departamento de Matem\'atica\\
   Universidade da Beira Interior\\
   6201-001 Covilh\~a\\
   Portugal}
\email{bento@ubi.pt}
\author[C\'esar M. Silva]{C\'esar M. Silva$^1$
   \protect\footnote{$^1$ C\MakeLowercase{orresponding author}}}
\address{C\'esar M. Silva\\
   Departamento de Matem\'atica\\
   Universidade da Beira Interior\\
   6201-001 Covilh\~a\\
   Portugal}
\email{csilva@ubi.pt}
\urladdr{www.mat.ubi.pt/~csilva}
\date{\today}
\begin{abstract}
   We establish the existence of local stable manifolds for semiflows generated
   by non\-lin\-ear per\-tur\-ba\-tions of nonau\-ton\-o\-mous ordinary linear
   differential equations in Banach spaces, assuming the existence of a general
   type of nonuniform dichotomy for the evolution operator that contains the
   nonuniform exponential and polynomial dichotomies as a very particular case.
   The family of dichotomies considered allow situations for which the
   classical Lyapunov exponents are zero. Additionally, we give new examples of
   application of our stable manifold theorem and study the behavior of the
   dynamics under perturbations.
\end{abstract}
\maketitle
\section{Introduction}
The concept of nonuniform hyperbolicity was introduced by
Pesin~\cite{Pesin-IANSSSR-1976,Pesin-UMN-1977,Pesin-IANSSSR-1977} and
generalizes the classical concept of (uniform) hyperbolicity by allowing the
rates of expansion and contraction to vary from point to point. For
nonuniformly hyperbolic trajectories, Pesin~\cite{Pesin-IANSSSR-1976} was able
to obtain a stable manifold theorem in the finite dimensional setting. Then, in
\cite{Ruelle-IHESPM-1979} Ruelle gave a proof of this theorem based on the
study of perturbations of products of matrices occurring in Oseledets'
multiplicative ergodic theorem~\cite{Oseledets-TMMS-1968}. Another proof, based
on the classical work of Hadamard, was obtained by Pugh and Shub
in~\cite{Pugh-Shub-TAMS-1989} and uses graph transform techniques. In the
infinite dimensional setting, Ruelle ~\cite{Ruelle-AM-1982} proved, following
his approach in \cite{Ruelle-IHESPM-1979}, a stable manifold theorem in Hilbert
spaces under some compactness assumptions. For transformations in Banach spaces
and under some compactness and invertibility assumptions, Ma\~n\'e established
the existence of stable manifolds in~\cite{Mane-LNM-1983} and
in~\cite{Thieullen-AIHPAN-1987} Thieullen weakened Ma\~n\'e's hypothesis.

In the context of nonautonomous differential equations, stable manifold
theorems were also obtained, assuming that the evolution operator have bounds
that are nonuniform, more precisely assuming that the evolution operator admits
a nonuniform exponential dichotomy, a notion introduced by Barreira and Valls
in~\cite{Barreira-Valls-JDE-2006} and inspired both in the classical notion of
exponential dichotomy introduced by Perron in~\cite{Perron-MZ-1930} and in the
notion of nonuniformly hyperbolic trajectory introduced by Pesin
in~\cite{Pesin-IANSSSR-1976,Pesin-UMN-1977,Pesin-IANSSSR-1977}. For more
details we refer the reader to the book~\cite{Barreira-Valls-LNM-2008}.

Recently it has been addressed the problem of obtaining stable manifolds for
perturbations of linear ordinary differential equations assuming the existence
of nonuniform dichotomies that are not exponential. Namely,
in~\cite{Bento-Silva-QJM-2012} were obtained local and global stable manifolds
for polynomial dichotomies and in~\cite{Bento-Silva-arXiv:0905.4935v1} it was proved the existence of global stable manifolds for a generalized type of dichotomy that
includes both the polynomial and exponential cases. Also in
\cite{Barreira-Valls-JFA-2009-257-(1018-1029)} Barreira and Valls obtained
local stable manifolds for perturbations of linear equations assuming a
dichotomy that follows growth rates of the form $\e^{\rho(t)}$ where
$\rho:\R^+_0 \to \R^+_0$ is an increasing differentiable function satisfying
\begin{equation}\label{eq1}
   \lim_{t \to +\infty} \frac{\log t}{\rho(t)} = 0.
\end{equation}
This definition, in spite of being very general, does not include some of the
growth rates considered in this paper, namely, due to~\eqref{eq1}, this
definition do not include for example the polynomial case studied in
~\cite{Bento-Silva-QJM-2012} and the growth rates of
Examples~\ref{ex:DicotomiaPolinomialLimite1}
and~\ref{ex:DicotomiaPolinomialLimite2}. In the discrete time setting, the
existence of global and local stable manifolds for perturbations of some
nonuniform polynomial dichotomies was discussed
in~\cite{Bento-Silva-JFA-2009}.

The main objective of this paper is to obtain local stable manifolds for
perturbations of nonautonomous linear ordinary differential equations, assuming
that the evolution operator associated with the linear equation admits a
dichotomy with growth rates given by increasing functions that go to infinity
(and therefore more general than the mentioned above). In fact we do not need
to assume condition~\eqref{eq1} and we allow growth rates given by
non-differentiable functions as well as different growth rates in the uniform
and in the nonuniform parts of the dichotomy. We also would like to emphasize
that the dichotomies considered here include as a particular case the ones
considered in~\cite{Barreira-Valls-JDE-2006} and~\cite{Bento-Silva-QJM-2012} and the theorems proved there, respectively, for nonuniform exponential dichotomies
and for nonuniform polynomial dichotomies, are particular cases of the result
presented in this paper. We also give new examples of growth rates to which our
local stable manifold theorem can be applied. We emphasize that the Lyapunov
exponent considered in \cite{Barreira-Valls-LNM-2008}, for Hilbert spaces, is
zero or infinity for most of the dichotomies considered in this paper.

The content of the paper is as follows: in Section~\ref{section:MR} we
establish the setting, we define the dichotomies and we state the main theorem;
in Section~\ref{section:ex} we give examples of nonuniform
$(\mu,\nu)$-dichotomies for each differentiable growth rates in our family of
growth rates and examples of growth rates that verify the conditions of the
main theorem;
in Section~\ref{section:proof} we prove the main theorem; finally, in
Section~\ref{section:perturbations}, we study how the manifolds obtained vary
with the perturbations considered.
\section{Main result}\label{section:MR}
Let $X$ be a Banach space and denote by $B(X)$ the space of bounded linear
operators acting on~$X$. Given a continuous function $A \colon \R^+_0 \to
B(X)$, we consider the initial value problem
\begin{equation} \label{eq:ivp-li}
   v' = A(t) v, \ v(s) = v_s
\end{equation}
with $s \ge 0$ and $v_s \in X$. We assume that each solution
of~\eqref{eq:ivp-li} is global and denote the evolution operator associated
with~\eqref{eq:ivp-li} by $T(t,s)$, i.e., $v(t) = T(t,s)v_s$ for $t \ge 0$. Note that the operator $T(t,s)$ is invertible (see Deimling \cite[Section 1.4]{Deimling-LNM_596-1977}).

We say that  an increasing function $\mu \colon \R^+_0 \to [1,+\infty[$ is a
growth rate if $\mu(0)=1$ and
   $$ \lim_{t \to + \infty} \mu(t) = +\infty.$$

Let $\mu$ and $\nu$ be growth rates. We say that equation \eqref{eq:ivp-li}
admits a \textit{nonuniform $(\mu,\nu)$-dichotomy} in $\R^+_0$ if, for each $t
\ge 0$, there are projections $P(t)$ such that
\begin{equation*}
   P(t)T(t,s) = T(t,s) P(s), \ t, s \ge 0
\end{equation*}
and constants $D \ge 1$, $a < 0 \le b$ and $\eps \ge 0$ such that, for every $t
\ge s \ge 0$,
\begin{align}
   & \| T(t,s)P(s)\|
      \le D \pfrac{\mu(t)}{\mu(s)}^a \nu(s)^\eps,\label{eq:dich-1}\\
   & \|T(t,s)^{-1} Q(t)\|
      \le D \pfrac{\mu(t)}{\mu(s)}^{-b} \nu(t)^\eps, \label{eq:dich-2}
\end{align}
where $Q(t)=\Id-P(t)$ is the complementary projection. When $\eps = 0$ we say
that we have a \textit{uniform $(\mu,\nu)$-dichotomy}  or simply a
\textit{$(\mu,\nu)$-dichotomy}.

For each $t \ge 0$, we define the linear subspaces
   $$ E(t)=P(t)X \quad \text{ and } \quad F(t)=Q(t)X.$$
Without loss of generality, we always identify the spaces $E(t) \times F(t)$
and $E(t) \oplus F(t)$ as the same vector space. The unique solution of~\eqref{eq:ivp-li}
can be written in the form
   $$ v(t) = \prts{U(t,s) \xi, V(t,s) \eta}, \ t \ge s$$
where $v_s = \prts{\xi, \eta} \in E(s) \times F(s)$ and
   $$ U(t,s) := P(t) T(t,s) P(s) \quad \text{ and } \quad
      V(t,s) := Q(t) T(t,s) Q(s).$$

In this paper we are going to address the problem of obtaining stable manifolds
for the nonlinear problem
\begin{equation} \label{eq:ivp-nonli}
   v' = A(t) v + f(t,v), \ v(s) = v_s
\end{equation}
when equation~\eqref{eq:ivp-li} admits a nonuniform $(\mu,\nu)$-dichotomy and
there are $c>0$ and $q>0$ such that the perturbations $f \colon \R^+_0 \times X
\to X$ verify, for $u,v \in X$ and $t \in \R_0^+$, the following conditions
\begin{align}
   & f(t,0)=0, \label{loc:cond-f-0}\\
   & \| f(t,u)- f(t,v)\|\le c \|u-v\| ( \|u\|+\|v\| )^q \label{loc:cond-f-2}.
\end{align}

Note that, making $v=0$ in~\eqref{loc:cond-f-2}, we have
\begin{equation} \label{loc:cond-f-3}
   \|f(t,u)\| \le c \|u\|^{q+1}.
\end{equation}
for every $u \in X$.

Writing the unique solution of \eqref{eq:ivp-nonli} in the form
   $$ (x(t,s,v_s),y(t,s,v_s)) \in E(t) \times F(t),$$
problem \eqref{eq:ivp-nonli} is equivalent to the following problem
\begin{align}
   & x(t) = U(t,s) \xi + \int_s^t U(t,r) f(r,x(r),y(r)) \dr,
      \label{eq:split-1a}\\
   & y(t) = V(t,s) \eta + \int_s^t V(t,r) f(r,x(r),y(r)) \dr,
   \label{eq:split-1b}
\end{align}
where $v_s = (\xi, \eta) \in E(s) \times F(s)$.

For each $(s,v_s)$ we consider the semiflow
\begin{equation} \label{def:Psi}
   \Psi_\tau(s,v_s)
   = \prts{s+\tau, x(s+\tau,s,v_s), y(s+\tau, s, v_s)}, \tau \ge0.
\end{equation}

Given $\delta > 0$ and a decreasing function $\beta \colon \R_0^+ \to \R_0^+$
we use the following notation
\begin{equation*}
   G_{\delta, \beta}
   = \bigcup\limits_{s \in \R^+_0} \set{s} \times B_{s,\delta,\beta},
\end{equation*}
where $B_{s,\delta,\beta}$ is the open ball of $E(s)$ centered at $0$ and with
radius $\delta \beta(s)$.

Let $\cX_{\delta,\beta}$ be the space of functions
   $$\phi \colon G_{\delta,\beta} \to X$$
such that, for every $(s,\xi), (s,\bar\xi) \in G_{\delta,\beta}$ the following
conditions hold
\begin{align}
   & \phi(s, \xi) \in F(s), \label{loc:cond-phi-1}\\
   & \phi(s,0)=0, \label{loc:cond-phi-0}\\
   & \| \phi(s,\xi) - \phi(s,\bar\xi)\|
      \le \|\xi-\bar\xi\| \label{loc:cond-phi-2}.
\end{align}
By~\eqref{loc:cond-phi-0}, putting $\bar\xi= 0$ in~\eqref{loc:cond-phi-2}, we
immediately conclude that
\begin{equation} \label{loc:cond-phi-3}
   \|\phi(s,\xi)\| \le \|\xi\|
\end{equation}
for every $(s,\xi) \in G_{\delta,\beta}$. We equip the space
$\cX_{\delta,\beta}$ with the metric defined by
\begin{equation} \label{metric:local:X}
   \|\phi - \psi\|'
   = \sup\set{\dfrac{\|\phi(s,\xi) - \psi(s,\xi)\|}{\|\xi\|}
      : (s, \xi) \in G_{\delta,\beta}, \ \xi \ne 0}
\end{equation}
for every $\phi, \psi \in \cX_{\delta,\beta}$. By~\eqref{loc:cond-phi-3} it
follows that, for every $(s,\xi) \in G_{\delta,\beta}$, $\|\phi(s,\xi)\| \le
\delta \beta(s) \le \delta \beta(0)$ and this implies that $\cX_{\delta,\beta}$
is a complete metric space with the metric given by~\eqref{metric:local:X}.

We also have to consider the space $\cX^*_{\delta,\beta}$ of continuous
functions $\phi : G \to X$, with
   $$ G=\bigcup_{s \in \R_0^+} \, \{s\} \times E(s),$$
such that $\phi(s,\xi) \in F(s)$ for every $(s,\xi) \in G$, the restriction
$\phi|_{G_{\delta,\beta}} \in \cX_{\delta,\beta}$ and
   $$ \phi(s,\xi)
      = \phi \prts{s, \dfrac{\delta \beta(s) \xi}{\|\xi\|}}
         \text{ whenever $\xi \not\in B_{s,\delta,\beta}$.}$$
Furthermore, since for every $\phi \in \cX_{\delta,\beta}$, we have a unique
Lipschitz extension of $\phi$ to $\bigcup_{t \ge 0} \set{t} \times
\overline{B_{t,\delta,\beta}}$, where $\overline{B_{t,\delta,\beta}}$ is the
closure of the ball $B_{t,\delta,\beta}$, there is a one-to-one correspondence
between $\cX_{\delta,\beta}$ and $\cX^*_{\delta,\beta}$. This one-to-one
correspondence allows us to define a metric in $\cX^*_{\delta,\beta}$ using the metric in $\cX_{\delta,\beta}$. Namely, this metric can be defined by
\begin{equation} \label{metric:local:X*}
   \|\phi - \psi\|'
   = \|(\phi - \psi)|_{G_{\delta,\beta}}\|',
\end{equation}
for every $\phi, \psi \in \cX^*_{\delta,\beta}$ and where the right hand side
is given by~\eqref{metric:local:X}. With this metric $\cX^*_{\delta,\beta}$ is
a complete metric space. Moreover, given $\phi, \psi \in \cX^*_{\delta,\beta}$,
it follows that
\begin{align}
   & \|\phi(s,\xi) - \phi(s,\bar\xi)\| \le 2 \|\xi - \bar\xi\|
         \label{loc:cond-phi-2-2}\\
   & \|\phi(s,\xi) - \psi(s,\xi)\| \le \|\phi - \psi\|' \|\xi\|
      \label{loc:cond-phi-2-3}
\end{align}
for every $(s,\xi), (s,\bar\xi) \in G$.

For every $\phi \in \cX_{\delta,\beta}$ we define the graph
\begin{equation} \label{def:V_phi,delta,beta}
   \cV_{\phi,\delta,\beta}
   = \set{(s,\xi, \phi(s,\xi)) : (s,\xi) \in G_{\delta,\beta}}.
\end{equation}

We now formulate our stable manifold theorem.

\begin{theorem} \label{thm:local}
   Given a Banach space $X$, let $f : \R^+_0 \times X \to X$ be a function
   satisfying \eqref{loc:cond-f-0} and \eqref{loc:cond-f-2} for some $c > 0$
   and $q > 0$. Suppose that equation~\eqref{eq:ivp-li} admits a nonuniform
   $(\mu,\nu)$-dichotomy in $\R^+_0$ for some growth rates $\mu$ and $\nu$, $D
   \ge 1$, $a < 0 \le b$ and $\eps \ge 0$. Assume that
   \begin{equation} \label{eq:CondicaoTeo}
      \lim_{t \to +\infty} \mu(t)^{a-b} \nu(t)^\eps = 0
   \end{equation}
   and
   \begin{equation} \label{loc-int-conv}
      \int_0^{+\infty} \mu(r)^{aq} \nu(r)^\eps \dr \text{ is convergent}.
   \end{equation}
   Define the functions $\beta,\tilde{\beta} \colon \R_0^+ \to \R_0^+$ by
   \begin{equation}\label{def:beta}
      \beta(t)
      = \dfrac{\mu(t)^a}{\nu(t)^{\eps(1+1/q)}
         \prts{\dint_t^{+\infty} \mu(r)^{aq} \nu(r)^\eps \dr}^{1/q}}.
   \end{equation}
   and $\tilde{\beta}(t)=\beta(t)\nu(t)^{-\eps}$ and suppose that
   \begin{equation}\label{loc:eq:decreasing}
      \beta(t) \text{ and } \mu(t)^a \beta(t)^{-1}
         \text{ are decreasing.}
   \end{equation}
   Then, for every $C > D$, choosing $\delta > 0$ sufficiently small,
   there is a unique $\phi \in \cX_{\delta,\beta}$ such that
   \begin{equation} \label{thm:local:invar}
      \Psi_\tau\prts{\cV_{\phi,\frac{\delta}{C},\tilde{\beta}}}
      \subset \cV_{\phi,\delta,\beta}
   \end{equation}
   for every $\tau \ge 0$, where $\Psi_\tau$ is given by~\eqref{def:Psi} and
   $\cV_{\phi,\frac{\delta}{C},\tilde{\beta}}$ and $\cV_{\phi,\delta,\beta}$
   are given by~\eqref{def:V_phi,delta,beta}. Furthermore, given $s \ge 0$, we
   have
   \begin{equation*}
      \| \Psi_{t-s}(p_{s,\xi}) - \Psi_{t-s}(p_{s,\bar\xi})\|
      \le 2 C \pfrac{\mu(t)}{\mu(s)}^a \nu(s)^\eps
         \|\xi - \bar\xi\|\\
   \end{equation*}
   for every $t \ge s$ and $\xi, \bar\xi \in
   B_{s,\frac{\delta}{C},\tilde\beta}$, where
   $p_{s,\xi} = (s, \xi, \phi(s,\xi))$.
\end{theorem}

We now use Theorem~\ref{thm:local} to establish the existence of stable manifolds for solutions of an ordinary differential equations that admit a nonuniform $\prts{\mu,\nu}$-hyperbolic behavior. Let $F \colon \R_0^+ \times X \to X$ be a $C^1$ function and let $v_0(t)$ be a solution of the equation
   \begin{equation} \label{eq:v'=F(t,v)}
      v' = F(t,v).
   \end{equation}
   Clearly, $v(t)$ is a solution of~\eqref{eq:v'=F(t,v)} if and only if $y(t) = v(t)-v_0(t)$ is a solution of
   \begin{equation}\label{eq:y'=A(t)y+f(t,y)}
      y' = A(t)y+f(t,y),
   \end{equation}
   where
   \begin{equation}\label{eq:A(t)=dF/dv(t,v_0(t))}
      A(t) = \dfrac{\partial F}{\partial v}(t,v_0(t))
   \end{equation}
   and
   \begin{equation}\label{eq:f(t,y)=F(t,y+v_0(t))-F(t,v_0(t))-A(t)y}
      f(t,y) = F(t,y+v_0(t))-F(t,v_0(t))-A(t)y.
   \end{equation}
   It is obvious that $f(t,0) = 0$ for every $t \in \R_0^+$. Moreover, since
      $$ \|f(t,y)-f(t,z)\|
         \le \sup_{\theta \in [0,1]}
            \norm{\dfrac{\partial f}{\partial y}(t, (1-\theta)y + \theta z)} \cdot \|y-z\|$$
   and
      $$ \dfrac{\partial f}{\partial y}(t,y)
         = \dfrac{\partial F}{\partial v}(t, y + v_0(t)) - A(t),$$
   if we suppose that
   \begin{equation}\label{eq:||dF/dv(t,y+v_0(t))-A(t)||<=c...}
      \norm{\dfrac{\partial F}{\partial v}(t,y+v_0(t)) - A(t)}
      \le c \|y\|^q \text{ \ \ for every $t \ge 0$ and every $y \in X$,}
   \end{equation}
   we have
   \begin{align*}
      \|f(t,y)-f(t,z)\|
      & \le c \sup_{\theta \in [0,1]}
         \|(1-\theta)y + \theta z)\|^q \cdot \|y-z\|\\
      & \le c \|y-z\| \prts{\|y\|+\|z\|}^q
   \end{align*}
   for every $t \in \R_0^+$ and for every $y,z \in X$. This proves that if~\eqref{eq:||dF/dv(t,y+v_0(t))-A(t)||<=c...} is satisfied, then the function $f$ defined by~\eqref{eq:f(t,y)=F(t,y+v_0(t))-F(t,v_0(t))-A(t)y} satisfies~\eqref{loc:cond-f-0} and~\eqref{loc:cond-f-2}.

   When $A(t)$ given by~\eqref{eq:A(t)=dF/dv(t,v_0(t))} admits a nonuniform $\prts{\mu,\nu}$-dichotomy in $\R_0^+$ we say that $v_0(t)$ is a \textit{nonuniform hyperbolic solution} of equation~\eqref{eq:v'=F(t,v)}. Therefore if $v_0(t)$ is a nonuniform hyperbolic solution of equation~\eqref{eq:v'=F(t,v)} and~\eqref{eq:||dF/dv(t,y+v_0(t))-A(t)||<=c...} is satisfied we can applied Theorem~\ref{thm:local} to equation~\eqref{eq:y'=A(t)y+f(t,y)} and we obtain immediately the following theorem analogous to Theorem 4 of Barreira and Valls~\cite{Barreira-Valls-JDE-2006}.

   \begin{theorem}\label{thm:trajectoria}
      Let $F \colon \R_0^+ \times X \to X$ be a function of class $C^1$ and let $v_0(t)$ be a nonuniform $\prts{\mu,\nu}$-hyperbolic solution of~\eqref{eq:v'=F(t,v)} such that~\eqref{eq:||dF/dv(t,y+v_0(t))-A(t)||<=c...} is satisfied for some $c , q > 0$ and conditions~\eqref{eq:CondicaoTeo},~\eqref{loc-int-conv} and~\eqref{loc:eq:decreasing} are satisfied. Then for every $C > D$, there is $\delta > 0$ and a unique $\phi \in \cX_{\delta,\beta}$ such that
      if $\prts{s,v_s} \in \cW_{\phi,\delta/C,\tilde\beta}$, then $\prts{t,v(t)} \in \cW_{\phi,\delta,\beta}$ for every $t \ge s$, where $v(t)=v(t,v_s)$ is the unique solution of~\eqref{eq:v'=F(t,v)} for $t \ge s$ with $v(s) = v_s$,
         $$ \cW_{\phi,\delta/C,\tilde\beta} = \set{\prts{s,\xi,\phi(s, \xi)}
            + (0,v_0(s)) \colon \prts{s,\xi} \in G_{\delta/C,\tilde\beta}},$$
      and
         $$ \cW_{\phi,\delta,\beta} = \set{\prts{s,\xi,\phi(s, \xi)}
            + (0,v_0(s)) \colon \prts{s,\xi} \in G_{\delta,\beta}}.$$
      Moreover,
         $$ \|v(t,v_s) - v(t,\overline v_s)\|
            \le C \pfrac{\mu(t)}{\mu(s)}^a \nu(s)^\eps
               \|v_s - \overline v_s\|$$
      for every $t \ge s$ and every $(s,v_s), (s,\overline{v}_s) \in \cW_{\phi,\delta/C,\tilde\beta}$.
   \end{theorem}
\section{Examples}\label{section:ex}
We start with an example of a nonautonomous linear equation that admits a
nonuniform $(\mu,\nu)$-dichotomy with arbitrary differentiable growth rates
$\mu$ and $\nu$.

\begin{example} \label{ex:dicho}
   Let $\eps > 0$ and $a < 0 \le b$. Put $\omega = \eps/2$ and let $\mu,\nu$
   be arbitrary differentiable growth rates. The differential
   equation in $\R^2$ given by
   \begin{equation} \label{eq:example}
      \begin{cases}
         u' = \prts{\dfrac{a \mu'(t)}{\mu(t)}
            + \dfrac{\omega \nu'(t)}{\nu(t)} (\cos t -1)
            - \omega \log \nu(t) \sin t} u\\[3mm]
         v' = \prts{\dfrac{b \mu(t)}{\mu'(t)}
            - \dfrac{\omega \nu'(t)}{\nu(t)} (\cos t -1)
            + \omega \log \nu(t) \sin t} v
      \end{cases}
   \end{equation}
   has the following evolution operator
      $$ T(t,s)(u,v) = (U(t,s)u, V(t,s)v),$$
   where
   \begin{align*}
      & U(t,s) = \pfrac{\mu(t)}{\mu(s)}^a
         e^{\omega \log \nu(t) (\cos t -1) - \omega \log \nu(s) (\cos s
         -1)},\\
      & V(t,s) = \pfrac{\mu(t)}{\mu(s)}^b
         e^{-\omega \log  \nu(t) (\cos t -1) + \omega \log \nu(s) (\cos s
         -1)}.
   \end{align*}
   Using the projections $P(t) \colon \R^2 \to \R^2$ defined by $P(t)(u,v)
   =(u,0)$ we have
   \begin{align*}
      & \| T(t,s)P(s)\|
         = \abs{U(t,s)}
         \le \pfrac{\mu(t)}{\mu(s)}^a \nu(s)^\eps\\
      & \|T(t,s)^{-1} Q(t)\|
         = \abs{V(t,s)^{-1}}
         \le \pfrac{\mu(t)}{\mu(s)}^{-b} \nu(t)^\eps
   \end{align*}
   and thus~\eqref{eq:example} admits a $(\mu,\nu)$-dichotomy. Moreover,
   if $t = 2 k \pi$ and $s = (2k -1) \pi$, $k \in \N$, then
      $$ U(t,s) = \pfrac{\mu(t)}{\mu(s)}^a \nu(s)^\eps$$
   and this ensures us that the nonuniform part can not be removed.
\end{example}

Now we will give examples of application of Theorem~\ref{thm:local}.

\begin{example}
   If $\mu(t) = \nu(t) = \e^t$, we get the local stable manifold theorem
   obtained by Barreira and Valls in~\cite{Barreira-Valls-JDE-2006}. In fact,
   in this case condition~\eqref{eq:CondicaoTeo} becomes $a + \eps < b$ and
   condition~\eqref{loc-int-conv} becomes $aq+\eps < 0$. The function $\beta$
   is given by
      $$ \beta(t)
         = \abs{a음 + \eps}^{1/q} \e^{- \eps (1 + 2/q)t}$$
   which is a decreasing function and
      $$ \mu(t) \beta(t)^{-1}
         =  \abs{a음 + \eps}^{1/q} \e^{(a + \eps(1+2/q))t}$$
   is decreasing if $a + \eps(1+2/q) < 0$.
\end{example}

\begin{example}
   Making $\mu(t) = \nu(t) = 1 +t$, we get the local stable manifold theorem
   obtained by the present authors in~\cite{Bento-Silva-QJM-2012}. In this case, condition~\eqref{eq:CondicaoTeo} becomes $a + \eps < b$ and
   condition~\eqref{loc-int-conv} becomes $aq+\eps +1 < 0$. Moreover,
      $$ \beta(t)
         = \abs{a음 + \eps + 1}^{1/q} (1+t)^{- (\eps (1 + 2/q) + 1/q)}$$
   is a decreasing function and
      $$ \mu(t) \beta(t)^{-1}
         =  \abs{a음 + \eps +1 }^{1/q} (1+t)^{(a + \eps(1+2/q)+1/q)}$$
   is a decreasing function if $a + \eps(1+2/q)+1/q < 0$.
\end{example}

In the next examples, we consider new growth rates for which
Theorem~\ref{thm:local} holds. Recall that Example~\ref{ex:dicho} allows us to
construct an example of a differential equation whose evolution operator has a
dichotomy with the growth rates given.

\begin{example}\label{ex:DicotomiaPolinomialLimite1}
   Consider
      $$ \mu(t) = (1+t) (1+\log(1+t))^\lbd
         \ \ \ \text{ and } \ \ \
         \nu(t) = 1+\log(1+t)$$
   with $\lbd > 0$. Then
      $$ \lim_{t \to +\infty} \mu(t)^{a-b} \nu(t)^\eps
         = \lim_{t \to +\infty} \, (1+t)^{a-b}
            (1+\log(1+t))^{(a-b)\lbd+\eps}
         = 0$$
   and
      $$ \int_0^{+\infty} \mu(t)^{aq} \nu(t)^\eps \, dt
         = \int_0^{+\infty} (1+t)^{aq} (1+\log(1+t))^{aq\lbd+\eps}\, dt$$
   is convergent if $aq < -1$ or $aq = -1 \ \wedge \ \eps - \lbd < -1$. When
   $aq=-1$  and  $\eps - \lbd < -1$ we have that
      $$ \beta(t)
         = (\lbd - \eps -1)^{1/q} (t+1)^{-1/q}
            (1 + \log(1+t))^{-\eps(1+2/q) - 1/q}$$
   is a decreasing function and
      $$ \mu(t)^a \beta(t)^{-1}
         = (\lbd - \eps -1)^{-1/q} (1 + \log(1+t))^{\eps(1+2/q) + 1/q -
         \lbd/q}$$
   is decreasing if \, $\eps(1+2/q) + 1/q - \lbd/q<0$. Hence, taking into
   account that $\eps(1+2/q) + 1/q - \lbd/q<0$ implies that $\eps - \lbd < -1$,
   for these growth rates we have a local stable manifold theorem if $aq=-1$
   and $\eps(1+2/q) + 1/q - \lbd/q<0$.
\end{example}

\begin{example}\label{ex:DicotomiaPolinomialLimite2}
   Let
      $$ \mu(t) = (1+t) (1+\log(1+t))(1+\log(1+(\log(1+t)))^\lbd$$
   and
      $$ \nu(t) = 1+\log(1+(\log(1+t))$$
   with $\lbd > 0$. Then
      \[ \begin{split}
            & \lim_{t \to +\infty} \mu(t)^{a-b} \nu(t)^\eps\\
            & = \lim_{t \to +\infty} \, (1+t)^{a-b} (1+\log(1+t))^{a-b}
                  (1+\log(1+\log(1+t)))^{(a-b)\lbd+\eps}\\
            & = 0
         \end{split}\]
   and
      $$ \int_0^{+\infty} \mu(t)^{aq} \nu(t)^\eps \, dt
         = \int_0^{+\infty} (1+t)^{aq} (1+\log(1+t))^{aq}
         (1+\log(1+\log(1+t)))^{aq\lbd+\eps}\, dt$$
   is convergent if $aq < -1$ or $aq = -1 \ \wedge \ \eps - \lbd < -1$. When
   $aq=-1$  and  $\eps - \lbd < -1$ we have that
      $$ \beta(t)
         = (\lbd - \eps -1)^{1/q} (t+1)^{-1/q}
            (1 + \log(1+t))^{-1/q} (1 + \log(1+\log(1+t)))^{-\eps(1+2/q) -
            1/q}$$
   is a decreasing function and
      $$ \mu(t)^a \beta(t)^{-1}
         = (\lbd - \eps -1)^{-1/q}
            (1 + \log(1+\log(1+t)))^{\eps(1+2/q) + 1/q - \lbd/q}$$
   is decreasing if \, $\eps(1+2/q) + 1/q - \lbd/q<0$. Hence, if $aq=-1$ and
   $\eps(1+2/q) + 1/q - \lbd/q<0$ we have again a local stable manifold
   theorem.
\end{example}

We now consider a family of differential equations to which Theorem~\ref{thm:trajectoria} can be applied.

\begin{example}
   Consider the nonautonomous system of ODEs
   \begin{equation}~\label{eq:x'=y(x-alpha(t))+alpha'(t),y'=x(y-beta(t))+beta'(t)}
   \begin{cases}
     x'=y(x-\alpha(t))+\alpha'(t)\\
     y'=x(y-\beta(t))+\beta'(t)
   \end{cases}
   \end{equation}
   where $\alpha$ and $\beta$ are $C^1$ functions, and assume that the solution $v_0(t)=(\alpha(t),\beta(t))$ of~\eqref{eq:x'=y(x-alpha(t))+alpha'(t),y'=x(y-beta(t))+beta'(t)} is a nonuniform $\prts{\mu,\nu}$-hyperbolic solution for some growth rates $\mu$ and $\nu$. Putting
      $$ F(t,x,y)=(y(x-\alpha(t))+\alpha'(t),x(y-\beta(t))+\beta'(t)),$$
   it is easy to see that, given $(x,y) \in \R^2$, we have
     $$ \norm{\dfrac{\partial F}{\partial v}(t,(x,y)+v_0(t))
            - \dfrac{\partial F}{\partial v}(t,v_0(t))}
         \le \sqrt{2\sqrt{2}} \norm{(x,y)},$$
   and~\eqref{eq:||dF/dv(t,y+v_0(t))-A(t)||<=c...} holds with $c=\sqrt{2\sqrt{2}}$ and $q=1$. Note that, if $\alpha(t)$ and $\beta(t)$ are the coefficients in the linear system~\eqref{eq:example}, we get a nonuniform $\prts{\mu,\nu}$-hyperbolic solution. Therefore, if conditions~\eqref{eq:CondicaoTeo},~\eqref{loc-int-conv} and~\eqref{loc:eq:decreasing} hold for the growth rates $\mu$ an $\nu$, the invariant manifolds and decay estimates in Theorem~\ref{thm:trajectoria} hold for the solution $v_0(t)=(\alpha(t),\beta(t))$ of system~\eqref{eq:x'=y(x-alpha(t))+alpha'(t),y'=x(y-beta(t))+beta'(t)}.
\end{example}
\section{Proof of Theorem~\ref{thm:local}}\label{section:proof}
From~\eqref{eq:split-1a} and~\eqref{eq:split-1b} we conclude that, to
prove~\eqref{thm:local:invar}, we must have
\begin{align}
   & x(t,\xi) = U(t,s) \xi  + \int_s^t U(t,r)
      f(r,x(r,\xi),\phi(r,x(r,\xi))) \dr, \label{eq:split-3a}\\
   & \phi(t,x(t,\xi)) = V(t,s) \phi(s,\xi) + \int_s^t V(t,r)
      f(r,x(r,\xi),\phi(r,x(r,\xi))) \dr \label{eq:split-3b}
\end{align}
for every $s \ge 0$, $t \ge s$ and $\xi \in
B_{s,\frac{\delta}{C},\tilde\beta}$. We are going to prove
that~\eqref{eq:split-3a} and~\eqref{eq:split-3b} hold using the Banach fixed
point theorem.

Thus, let $\cB_{s,\delta,\beta}$ be the space of functions
   $$ x : [s,+\infty[ \times B_{s,\delta,\beta} \to X$$
such that, for every $t \ge s$ and $\xi, \bar\xi \in B_{s,\delta,\beta}$, we
have
\begin{align}
   & x(t,\xi) \in E(t), \label{loc:cond-x-in-E(t)}\\
   & x(s, \xi) = \xi, \ x(t,0) = 0, \label{loc:cond-x-0}\\
   & \| x(t,\xi)- x(t,\bar\xi)\|
      \le C \pfrac{\mu(t)}{\mu(s)}^a \nu(s)^{\eps}
         \|\xi - \bar \xi\|. \label{loc:cond-x-1}
\end{align}
Setting $\bar\xi =0$ in~\eqref{loc:cond-x-1} we obtain the following estimate
\begin{equation}\label{loc:cond-x-2}
   \| x(t,\xi)\| \le C \pfrac{\mu(t)}{\mu(s)}^a \nu(s)^{\eps} \|\xi\|
   \le C \delta \pfrac{\mu(t)}{\mu(s)}^a \nu(s)^\eps \beta(s).
\end{equation}
We equip $\cB_{s,\delta,\beta}$ with the metric defined by
\begin{equation} \label{metric_B_s,delta,beta}
   \|x - y\|'
   = \sup\set{\dfrac{\|x(t,\xi) - y(t,\xi)\|}{\|\xi\|}
      \pfrac{\mu(t)}{ \mu(s)}^{-a}\nu(s)^{-\eps}
      \colon t \ge s, \ \xi \in B_{s,\delta,\beta}\setm{0}}
\end{equation}
for every $x, y \in \cB_{s,\delta,\beta}$. With this metric
$\cB_{s,\delta,\beta}$ is a complete metric space.

Given $x \in \cB_{s,\delta,\beta}$  and $\phi \in \cX^*_{\delta,\beta}$ we use
the following notation
\begin{equation*}
   \phi_x(r,\xi) = \phi(r,x(r,\xi)) \text{ and }
   f_{x,\phi}(r,\xi) = f(t,x(r,\xi),\phi_x(r,\xi)).
\end{equation*}

\begin{lemma} \label{lemma:local:aux1}
   For every $\phi \in \cX^*_{\delta,\beta}$, choosing $\delta > 0$
   sufficiently small, there is one and only one $x = x_\phi \in
   \cB_{s,\delta,\beta}$ such that
   \begin{equation} \label{loc:eq:x=U+int}
       x(t,\xi)
       = U(t,s) \xi + \int_s^t U(t,r) f_{x,\phi}(r,\xi) \dr
   \end{equation}
   for every $t \ge s$ and $\xi \in B_{s,\delta,\beta}$. Moreover, choosing
   $\delta >0$ sufficiently small, we have
   \begin{equation} \label{loc:eq:norm_x_phi-x_psi}
      \|x_\phi(t,\xi) - x_\psi(t,\xi)\|
      \le C \pfrac{\mu(t)}{\mu(s)}^a \|\xi\| \cdot \|\phi - \psi\|'
   \end{equation}
   for every $\phi, \psi \in \cX^*_{\delta,\beta}$, every $t \ge s$ and every
   $\xi \in B_{s,\delta,\beta}$.
\end{lemma}

\begin{proof}
   In $\cB_{s,\delta,\beta}$ we define an operator $J = J_{\phi}$ by
      $$ (Jx)(t,\xi) = U(t,s) \xi + \int_s^t U(t,r) f_{x,\phi}(r,\xi) \dr.$$
   Obviously, $(Jx)(s,\xi) = \xi$ for every $\xi \in B_{s,\delta,\beta}$ and
   from~\eqref{loc:cond-x-0},~\eqref{loc:cond-phi-0} and~\eqref{loc:cond-f-0}
   it follows that $(Jx)(t,0) = 0$ for every $t \ge s$. Moreover, $Jx$
   satisfies~\eqref{loc:cond-x-in-E(t)} for every $t \ge s$ and every $\xi \in
   B_{s,\delta,\beta}$.

   By~\eqref{loc:cond-f-2} and~\eqref{loc:cond-phi-2-2} it follows for every $r
   \ge s$ and every $\xi \in B_{s,\delta,\beta}$ that
   \begin{align*}
      \| f_{x,\phi}(r,\xi) - f_{x,\phi}(r,\bar\xi) \|
      & \le c \left( \| x(r,\xi)-x(r,\bar\xi) \|
         + \| \phi_x(r,\xi) - \phi_x(r,\bar\xi) \| \right) \times \\
      & \phantom{\le} \times (\|x(r,\xi)\|+\|\phi_x(r,\xi)\|+\|x(r,\bar\xi)\|
         +\| \phi_x(r,\bar\xi)\| )^q \\
      & \le 3^{q+1} c \, \|x(r,\xi)-x(r,\bar\xi)\|
         \left( \|x(r,\xi)\| + \|x(r,\bar\xi)\| \right)^q
   \end{align*}
   and by~\eqref{loc:cond-x-1} and~\eqref{loc:cond-x-2} we get the following
   estimate
   \begin{equation} \label{loc:eq:norm:der:f-f}
      \| f_{x,\phi}(r,\xi) - f_{x,\phi}(r,\bar\xi) \|
      \le 2^q 3^{q+1} c \, C^{q+1} \delta^q \pfrac{\mu(r)}{\mu(s)}^{\! aq+a}
         \nu(s)^{\eps(q+1)} \beta(s)^q \| \xi-\bar\xi \|.
   \end{equation}
   From~\eqref{eq:dich-1}, the last estimate and because by~\eqref{def:beta} we
   have
   \begin{equation} \label{loc:equa:fund}
      \mu(s)^{-aq} \nu(s)^{\eps(q+1)} \beta(s)^q \dint_s^{+\infty}
      \mu(r)^{aq}\nu(r)^{\eps} \dr
      = 1,
   \end{equation}
   we obtain the following estimate
   \begin{align*}
      & \int_s^t \|U(t,r)\| \cdot
         \|f_{x,\phi}(r,\xi)-f_{x,\phi}(r,\bar\xi)\| \dr\\
      & \le 2^q 3^{q+1} c \, C^{q+1} D \delta^q \pfrac{\mu(t)}{\mu(s)}^a
         \mu(s)^{- aq} \nu(s)^{\eps(q+1)} \beta(s)^q \|\xi-\bar\xi\|
         \int_s^t \mu(r)^{aq} \nu(r)^{\eps} \dr \\
      & \le 2^q 3^{q+1} c \, C^{q+1} D \delta^q
         \pfrac{\mu(t)}{\mu(s)}^a \|\xi-\bar\xi\|
   \end{align*}
   and using again~\eqref{eq:dich-1} it follows that
   \begin{align*}
      & \|(J_\phi x)(t,\xi) - (J_\phi x)(t,\bar\xi)\|\\
      & \le \|U(t,s)\| \|\xi-\bar\xi\| + \int_s^t \|U(t,r)\|
         \cdot \|f_{x,\phi}(r,\xi)-f_{x,\phi}(r,\bar\xi)\| \dr \\
      & \le \prts{D + 2^q 3^{q+1} c \, C^{q+1} D \delta^q}
         \pfrac{\mu(t)}{\mu(s)}^a \nu(s)^\eps \|\xi-\bar\xi\|.
   \end{align*}
   Therefore, since $C > D$, choosing $\delta>0$ sufficiently small we have
      $$ \|(J_\phi x)(t,\xi) - (J_\phi x)(t,\bar\xi)\|
         \le C \left( \frac{\mu(t)}{\mu(s)} \right)^a \nu(s)^\eps
         \|\xi-\bar\xi\|,$$
   and this implies the inclusion $J(\cB_{s,\delta,\beta}) \subseteq
   \cB_{s,\delta,\beta}$.

   Now we are going to prove that if $\delta$ is sufficiently small $J_\phi$ is
   a contraction. Given $x, y \in \cB_{s,\delta,\beta}$,
   from~\eqref{loc:cond-f-2} and~\eqref{loc:cond-x-2} we have for every $r \ge
   s$ and every $\xi \in B_{s,\delta,\beta}$
   \begin{align*}
      \|f_{x,\phi}(r,\xi) - f_{y,\phi}(r,\xi)\|
      & \le c (\|x(r,\xi)-y(r,\xi)\|+\|\phi_x(r,\xi)-\phi_y(r,\xi)\|) \times\\
      & \hspace{1cm} \times (\| x(r,\xi)\| + \|\phi_x(r,\xi)\|
         +\|y(r,\xi)\| + \|\phi_y(r,\xi) \|)^q \\
      & \le 3^{q+1} c \|x(r,\xi)-y(r,\xi)\| (\|x(r,\xi)\|+\|y(r,\xi)\|)^q \\
      & \le 2^q 3^{q+1} c \, C^q \delta^q \pfrac{\mu(r)}{\mu(s)}^{aq}
         \nu(s)^{\eps q} \beta(s)^q \| x(r,\xi) - y(r,\xi)\|\\
      & \le 2^q 3^{q+1} c \, C^q \delta^q \pfrac{\mu(r)}{\mu(s)}^{aq+a}
         \nu(s)^{\eps (q+1)} \beta(s)^q \|x - y\|' \|\xi\|
   \end{align*}
   and by~\eqref{eq:dich-1} and~\eqref{loc:equa:fund} it follows that
   \begin{align*}
      & \|(J_\phi x)(t,\xi)-(J_\phi y)(t,\xi)\| \\
      & \le \int_s^t \|U(t,r)\| \|f_{x,\phi}(r,\xi)-f_{y,\phi}(r,\xi)\| \dr\\
      & \le 2^q 3^{q+1} c \, C^q D \delta^q \pfrac{\mu(t)}{\mu(s)}^a
         \mu(s)^{-aq} \nu(s)^{\eps (q+1)} \beta(s)^q \|x - y\|' \|\xi\|
         \int_s^t \mu(r)^{aq}\nu(r)^{\eps} \dr \\
      & \le 2^q 3^{q+1} c \, C^q D \delta^q \pfrac{\mu(t)}{\mu(s)}^a
         \|x - y\|' \|\xi\|.
   \end{align*}
   Therefore
      $$ \|J_\phi x - J_\phi y\|'
         \le 2^q 3^{q+1} c \, C^q D \delta^q \|x-y\|'$$
   and choosing $\delta > 0$ sufficiently small, $J_\phi$ is a contraction. By
   the Banach fixed point theorem, $J_\phi$ has a unique fixed point $x_\phi
   \in \cB_{s,\delta,\beta}$ and $x_\phi$ verifies $\eqref{loc:eq:x=U+int}$.

   Now we will prove~\eqref{loc:eq:norm_x_phi-x_psi}. Let $\phi, \psi \in
   \cX_{\delta,\beta}^*$ and let $z \in \cB_{s,\delta,\beta}$ be defined by
   $z(t,\xi) = U(t,s)\xi$ for every $t \ge s$ and every $\xi \in
   B_{s,\delta,\beta}$. Putting $z_1 = y_1 = z$ and $y_{n+1} = J_\phi y_n$ and
   $z_{n+1} = J_\psi z_n$ for each $n \in \N$, we have
      $$ \|x_\phi - x_\psi\|'
         = \lim_{n \to \infty} \|y_n - z_n\|'.$$
   Hence, to prove~\eqref{loc:eq:norm_x_phi-x_psi} it is enough to prove that
   for each $n \in \N$ we have
   \begin{equation} \label{loc:norm:y_n-z_n}
      \|y_n(t,\xi) - z_n(t,\xi)\|
      \le C \pfrac{\mu(t)}{\mu(s)}^a  \|\xi\| \cdot \|\phi - \psi\|'
   \end{equation}
   for every $t \ge s$  and every $\xi \in B_{s,\delta,\beta}$. We are going
   to
   prove~\eqref{loc:norm:y_n-z_n} by mathematical induction on $n$. For $n=1$
   there is nothing to prove. Suppose that~\eqref{loc:norm:y_n-z_n} is
   true for $n$. Then
   from~\eqref{loc:cond-f-2},~\eqref{loc:cond-phi-2-2},~\eqref{loc:cond-x-2},
   ~\eqref{loc:cond-phi-2-3} and~\eqref{loc:norm:y_n-z_n} we have, since
   $\nu(s) \ge 1$,
   \begin{align*}
      & \|f_{y_n,\phi}(r,\xi) - f_{z_n,\psi}(r,\xi)\| \\
      & \le c \prts{\|y_n(r,\xi) - z_n(r,\xi)\|
         + \|\phi_{y_n} (r,\xi) - \psi_{z_n} (r,\xi)\|} \times \\
      & \phantom{ \le } \ \ \ \ \ \ \times
         \prts{\|y_n(r,\xi)\| + \|\phi_{y_n} (r,\xi)\|
         + \|z_n(r,\xi)\| + \| \psi_{z_n} (r,\xi)\|}^q \\
      & \le 3^q c \prts{3 \|y_n(r,\xi) - z_n(r,\xi)\| +
         \|\phi_{z_n} (r,\xi) - \psi_{z_n} (r,\xi)\|}
         \prts{\|y_n(r,\xi)\| + \|z_n(r,\xi)\|}^q\\
      & \le 3^q c \prts{3 \|y_n(r,\xi) - z_n(r,\xi)\|
         + \|\phi - \psi\|' \|z_n(r,\xi)\|}
         \prtsr{2 C \delta \pfrac{\mu(r)}{\mu(s)}^a \nu(s)^\eps \beta(s)}^q\\
      & \le 2^{q+2} 3^q c \, C^{q+1} \delta^q \pfrac{\mu(r)}{\mu(s)}^{aq+a}
         \nu(s)^{\eps(q+1)} \beta(s)^q \ \|\xi\| \cdot \|\phi - \psi\|'
   \end{align*}
   and this implies that
   \begin{align*}
      & \|y_{n+1}(t,\xi) - z_{n+1}(t,\xi)\|\\
      & \le \int_s^t \|U(t,r)\|
         \|f_{y_n,\phi}(r,\xi) - f_{z_n,\psi}(r,\xi)\| \dr\\
      & \le K \pfrac{\mu(t)}{\mu(s)}^a \mu(s)^{-aq} \nu(s)^{\eps(q+1)}
         \beta(s)^q \|\xi\| \cdot \|\phi - \psi\|'
         \int_s^t \mu(r)^{aq}\nu(r)^{\eps} \dr\\
      & \le K \pfrac{\mu(t)}{\mu(s)}^a \ \|\xi\| \cdot \|\phi - \psi\|'
   \end{align*}
   with $K = 2^{q+2} 3^q c \, C^{q+1} D \delta^q$. Choosing $\delta > 0$ such
   that $\delta^q < \dfrac{1}{2^{q+2} 3^q c \, C^q D}$, we have
      $$ \|y_{n+1}(t,\xi) - z_{n+1}(t,\xi)\|
         \le C \pfrac{\mu(t)}{\mu(s)}^a \|\xi\| \cdot \|\phi - \psi\|'.$$
   Therefore~\eqref{loc:norm:y_n-z_n} is true for every $n \in \N$ and this
   completes the proof of the lemma.
\end{proof}

\begin{lemma} \label{lemma:local:equiv}
   If $\delta > 0$ is sufficiently small and  $\phi \in \cX_{\delta,\beta}^*$,
   denoting by $x_\phi$ the unique function given by
   Lemma~\ref{lemma:local:aux1}, the following properties hold
   \begin{enumerate}[$a)$]
      \item if the identity
         \begin{equation} \label{loc:eq:phi_x}
            \phi_{x_\phi} (t,\xi)
            = V(t,s) \phi(s,\xi)
               + \int_s^t V(t,r) f_{x_\phi, \phi} (r,\xi) \dr
         \end{equation}
         holds for every $s\ge 0$, $t \ge s$ and $\xi \in
         B_{s,\delta,\beta}$, then
         \begin{equation} \label{loc:eq:phi}
            \phi(s,\xi)
            =  - \int_s^{+\infty} V(r,s)^{-1} f_{x_\phi, \phi} (r,\xi)
            \dr
         \end{equation}
         for every $s \ge 0$ and every $\xi \in B_{s,\delta,\beta}$;
      \item if~\eqref{loc:eq:phi} holds for every $s \ge 0$ and every $\xi
          \in B_{s,\delta,\beta}$, then~\eqref{loc:eq:phi_x} holds for
          every $s \ge 0$ and every $\xi \in
          B_{s,\frac{\delta}{C},\tilde{\beta}}$.
   \end{enumerate}
\end{lemma}

\begin{proof}
   First we prove that the integral in~\eqref{loc:eq:phi} is convergent.
   From~\eqref{loc:cond-f-3} and~\eqref{loc:cond-x-2} we have
   \begin{align*}
      \|f_{x_\phi,\phi}(r,\xi)\|
      & \le c \prts{\|x_\phi(r,\xi)\| + \|\phi_{x_\phi}(r,\xi)\|}^{q+1}\\
      & \le 3^{q+1} c \, \|x_\phi(r,\xi)\|^{q+1}\\
      & \le 3^{q+1} c \, C^{q+1} \delta^{q+1}
         \pfrac{\mu(r)}{\mu(s)}^{aq+a} \nu(s)^{\eps(q+1)} \beta(s)^{q+1}
   \end{align*}
   and, since $a-b<0$, this implies that
   \begin{align*}
      & \int_s^{+\infty} \|V(r,s)^{-1}\| \, \|f_{x_\phi, \phi} (r,\xi)\| \dr\\
      & \le 3^{q+1} c \, C^{q+1} D \delta^{q+1}
         \mu(s)^{-aq-a+b}\nu(s)^{\eps(q+1)}\beta(s)^{q+1}
         \int_s^{+\infty} \mu(r)^{aq+a-b}\nu(r)^{\eps} \dr\\
      & \le 3^{q+1} c \, C^{q+1} D \delta^{q+1}
         \mu(s)^{-aq} \nu(s)^{\eps(q+1)}\beta(s)^{q+1}
         \int_s^{+\infty} \mu(r)^{aq}\nu(r)^{\eps} \dr
   \end{align*}
   and, by~\eqref{loc-int-conv}, the integral is convergent.

   If~\eqref{loc:eq:phi_x} is true for every $s\ge 0$, $t \ge s$ and $\xi \in
   B_{s,\delta,\beta}$ then
   \begin{equation} \label{loc:eq:phi-aux1}
      \phi(s,\xi)
      = V(t,s)^{-1} \phi_{x_\phi} (t,\xi)- \int_s^t V(r,s)^{-1}
         f_{x_\phi, \phi} (r,\xi) \dr.
   \end{equation}
   From~\eqref{eq:dich-2},~\eqref{loc:cond-x-2} and~\eqref{eq:CondicaoTeo} we
   have
      $$ \lim_{t \to +\infty} \|V(t,s)^{-1} \phi_{x_\phi} (t,\xi)\|
         \le \lim_{t \to +\infty} D \pfrac{\mu(t)}{\mu(s)}^{-b} \nu(t)^\eps
            2 C \pfrac{\mu(t)}{\mu(s)}^a \nu(s)^\eps \|\xi\|
         =0.$$
   Thus, letting $t \to + \infty$ in~\eqref{loc:eq:phi-aux1}, we can conclude
   that~\eqref{loc:eq:phi} holds for every $s \ge 0$ and every $\xi \in
   B_{s,\delta,\beta}$. Therefore $a)$ holds.

   Suppose now that~\eqref{loc:eq:phi} holds for every $s \ge 0$ and every $\xi
   \in B_{s,\delta,\beta}$. Defining, for each $(s,\xi) \in G_{\delta,\beta}$,
   a semiflow by
      $$ F_r(s,\xi)=(s+r, x_\phi(s+r,\xi)), \ r \ge 0$$
   from~\eqref{loc:eq:phi}, we have
   \begin{equation} \label{loc:eq:phi-aux2}
      \phi(s,\xi)
      = - \int_s^{+\infty} V(r,s)^{-1}
         f(F_{r-s}(s,\xi),\phi(F_{r-s}(s,\xi)))\dr.
   \end{equation}
   If $\xi \in B_{s,\delta / C,\tilde{\beta}}$, then by
   \eqref{loc:eq:decreasing} we get
      $$ \|x_\phi(t,\xi)\|
         \le C \pfrac{\mu(t)}{\mu(s)}^a \nu(s)^\eps \dfrac{\delta}{C} \tilde\beta(s)
         = \delta \dfrac{\mu(t)^a \beta(t)^{-1}}{\mu(s)^a \beta(s)^{-1}} \beta(t)
         \le \delta \beta(t)$$
   and this means that $x_\phi(t,\xi) \in B_{t,\delta,\beta}$. Moreover, since
      $$ F_{r-t}(t,x_\phi(t,\xi))
         = F_{r-t}(F_{t-s}(s,\xi))
         = F_{r-s}(s,\xi)
         = (r,x_\phi(r,\xi)),$$
   replacing $(s,\xi)$ by $(t,x_\phi(t,\xi))$ in~\eqref{loc:eq:phi-aux2} we
   have
   \begin{align*}
      \phi(t,x_\phi(t,\xi))
      & = - \int_t^{+\infty} V(r,t)^{-1}
         f(F_{r-t}(t,x_\phi(t,\xi)),\phi(F_{r-t}(t,x_\phi(t,\xi))))\dr\\
      & = - \int_t^{+\infty} V(r,t)^{-1}
         f(r,x_\phi(r,\xi),\phi(r,x_\phi(r,\xi)))\dr\\
      & = - \int_t^{+\infty} V(r,t)^{-1} f_{x_\phi,\phi}(r,\xi)\dr.
   \end{align*}
   Then, using the fact that $V(t,s) V(r,s)^{-1} = V(t,r)$, we have
   by~\eqref{loc:eq:phi}
      $$ V(t,s) \phi(s,\xi)
         = - \int_s^{+\infty} V(t,s) V(r,s)^{-1} f_{x,\phi} (r,\xi) \dr$$
   and this is equivalent to
   \begin{align*}
      V(t,s) \phi(s,\xi) + \int_s^t V(t,r) f_{x,\phi}(r,\xi)\dr
      & = -\int_t^{+\infty} V(r,t)^{-1} f_{x,\phi}(r,\xi)\dr\\
      & = \phi(t,x_\phi(t,\xi)).
   \end{align*}
  This finishes the proof of the lemma.
\end{proof}

\begin{lemma} \label{lemma:local:aux2}
   Choosing $\delta > 0$ sufficiently small, there is a unique $\phi \in
   \cX^*_{\delta,\beta}$ such that~\eqref{loc:eq:phi} holds for every $s \ge 0$
   and every $\xi \in E(s)$.
\end{lemma}

\begin{proof}
   Define in $\cX^*_{\delta,\beta}$ the operator $\Phi$ by
      $$ \prts{\Phi \phi}(s,\xi)
         = - \int_s^{+\infty} V(r,s)^{-1} f_{x_\phi, \phi} (r,\xi) \dr$$
   for every $(s,\xi) \in G_{\delta,\beta}$ (and this can be extended uniquely
   by continuity to the closure of $G_{\delta,\beta}$) and by
      $$ \prts{\Phi \phi}(s,\xi)
         = \prts{\Phi \phi} \prts{s, \dfrac{\delta \beta(s) \xi}{\|\xi\|}}.$$
   for every $(s,\xi) \not\in G_{\delta,\beta}$. It follows immediately from
   the
   definition of $\Phi$ that $\Phi \phi$ satisfies~\eqref{loc:cond-phi-1}.
   Furthermore, from~\eqref{loc:cond-x-0},~\eqref{loc:cond-phi-0}
   and~\eqref{loc:cond-f-0}, we have $\prts{\Phi \phi}(s,0) = 0$ for every $s
   \ge 0$.

   Given $(s,\xi), (s,\bar\xi) \in G_{\delta,\beta}$, from~\eqref{eq:dich-2},
   \eqref{loc:eq:norm:der:f-f} and~\eqref{loc:equa:fund}, we have
   \begin{align*}
      & \|(\Phi \phi)(s,\xi) - (\Phi \phi)(s,\bar\xi)\|\\
      & \le \int_s^{+\infty} \|V(r,s)^{-1}\| \cdot
         \|f_{x_\phi, \phi} (r,\xi) - f_{x_\phi, \phi} (r,\bar\xi)\| \dr\\
      & \le 2^q 3^{q+1} c \, C^{q+1} D \delta^q
         \mu(s)^{- aq - a + b} \nu(s)^{\eps(q+1)}\beta(s)^q \|\xi - \bar\xi\|
         \int_s^{+\infty} \mu(r)^{aq + a - b} \nu(r)^\eps \dr\\
      & \le 2^q 3^{q+1} c \, C^{q+1} D \delta^q
         \mu(s)^{- aq} \nu(s)^{\eps(q+1)}\beta(s)^q \|\xi - \bar\xi\|
         \int_s^{+\infty} \mu(r)^{aq} \nu(r)^\eps \dr\\
      & = 2^q 3^{q+1} c \, C^{q+1} D \delta^q \|\xi - \bar\xi\|.
   \end{align*}
   Choosing $\delta$ sufficiently small we get
      $$ \|(\Phi \phi)(s,\xi) - (\Phi \phi)(s,\bar\xi)\|
         \le \|\xi-\bar\xi \|.$$
   Therefore $\Phi\prts{\cX^*_{\delta,\beta}} \subseteq \cX^*_{\delta,\beta}$.

   Now we prove that $\Phi$ is a contraction in $\cX^*_{\delta,\beta}$ with the
   metric given by~\eqref{metric:local:X*}. Let $\phi, \psi \in
   \cX^*_{\delta,\beta}$ and $s \ge 0$. Then, since for every $r \ge s$ and
   every $\xi \in B_{s,\delta,\beta}$ we have
   \begin{align*}
      & \|f_{x_\phi,\phi}(r,\xi) - f_{x_\psi,\psi}(r,\xi)\|\\
      & \le c \prts{\|x_\phi(r,\xi) - x_\psi(r,\xi)\|
         + \|\phi_{x_\phi}(r,\xi) - \psi_{x_\psi}(r,\xi)\|} \times\\
      & \hspace*{1cm} \times \prts{\|x_\phi(r,\xi)\| + \|x_\psi(r,\xi)\|
         + \|\phi_{x_\phi}(r,\xi)\| + \|\psi_{x_\psi}(r,\xi)\|}^q\\
      & \le 3^q c \prts{3 \|x_\phi(r,\xi) - x_\psi(r,\xi)\|
         + \|\phi_{x_\phi}(r,\xi) - \psi_{x_\phi}(r,\xi)\|}
         \prts{\|x_\phi(r,\xi)\| + \|x_\psi(r,\xi)\|}^q\\
      & \le 3^q c \prts{3 \|x_\phi(r,\xi) - x_\psi(r,\xi)\|
         + \|\phi- \psi\|' \|x_\phi(r,\xi)\|}
            \prtsr{2 C \delta \pfrac{\mu(r)}{\mu(s)}^a \nu(s)^{\eps}
            \beta(s)}^q
   \end{align*}
   it follows from~\eqref{loc:eq:norm_x_phi-x_psi} that
   \begin{align*}
      \|f_{x_\phi,\phi}(r,\xi) & - f_{x_\psi,\psi}(r,\xi)\|\\
      & \le 2^{q+2} 3^q c \, C^{q+1} \delta^q
         \pfrac{\mu(r)}{\mu(s)}^{aq+a} \nu(s)^{\eps(q+1)}\beta(s)^q
         \ \|\xi\| \cdot \|\phi - \psi\|'
   \end{align*}
   for every $r \ge s$ and every $\xi \in B_{s,\delta,\beta}$.
   From~\eqref{eq:dich-2}, \eqref{eq:CondicaoTeo}
   and~\eqref{loc-int-conv} we have
   \begin{align*}
      & \|(\Phi \phi)(s,\xi) - (\Phi \psi) (s,\xi)\|\\
      & \le \int_s^{+\infty} \|V(r,s)^{-1}\| \cdot
         \|f_{x_\phi,\phi}(r,\xi) - f_{x_\psi,\psi}(r,\xi)\| \dr\\
      & \le K \mu(s)^{-aq-a+b} \nu(s)^{\eps(q+1)} \beta(s)^q
         \|\xi\| \cdot \|\phi - \psi\|'
         \int_s^{+\infty} \mu(r)^{aq+a-b} \nu(r)^\eps \dr\\
      & \le K \mu(s)^{-aq} \nu(s)^{\eps(q+1)} \beta(s)^q
         \|\xi\| \cdot \|\phi - \psi\|'
         \int_s^{+\infty} \mu(r)^{aq} \nu(r)^\eps \dr\\
      & =  K \|\xi\| \cdot \|\phi - \psi\|'
   \end{align*}
   for every $s \ge 0$ and every $\xi \in B_{s,\delta,\beta}$ and with $K =
   2^{q+2} 3^q c \, C^{q+1} D \delta^q$. Therefore, for every $\phi, \psi \in
   \cX^*_{\delta,\beta}$, we obtain
      $$ \|(\Phi \phi) - (\Phi \psi) \|'
         \le  2^{q+2} 3^q c \, C^{q+1} D \delta^q \ \|\phi - \psi\|'$$
    and choosing $\delta > 0$ sufficiently small, $\Phi$ is a contraction on
    $\cX^*_{\delta,\beta}$.

   By the Banach fixed point theorem, $\Phi$ has a unique fixed point and this
   fixed point verifies~\eqref{loc:eq:phi} for every $s \ge 0$ and every $\xi
   \in E(s)$.
\end{proof}

Now the proof of Theorem~\ref{thm:local} follows easily.

\begin{proof}[Proof of Theorem~\ref{thm:local}]
   For each $\phi \in \cX^*_{\delta,\beta}$, using
   Lemma~\ref{lemma:local:aux1}, there is a unique function $x_\phi$ in
   $\cB_{s,\delta,\beta}$ satisfying~\eqref{eq:split-3a}. By
   Lemma~\ref{lemma:local:equiv} solving~\eqref{eq:split-3b} is equivalent to
   solve~\eqref{loc:eq:phi} and from Lemma~\ref{lemma:local:aux2} there is a
   unique solution of~\eqref{loc:eq:phi}. Hence, choosing $\delta>0$
   sufficiently small, the existence of a stable manifold is established.

   Moreover, for every $s \ge 0$, every $t \ge s$  and every $\xi,\bar\xi \in
   B_{s,\frac{\delta}{C},\tilde{\beta}}$, it follows
   from~\eqref{loc:cond-phi-2}
   and~\eqref{loc:cond-x-1} that
   \begin{align*}
      \| \Psi_{t-s}(p_{s,\xi}) - \Psi_{t-s} (p_{s,\bar\xi})\|
      & = \|\prts{t, x_\phi(t,\xi), \phi_{x_\phi} (t,\xi)}
         - \prts{t, x_\phi(t,\bar\xi), \phi_{x_\phi} (t,\bar\xi)}\|\\
      & \le \|x_\phi(t,\xi) - x_\phi(t,\bar\xi)\|
         + \|\phi_{x_\phi}(t,\xi) - \phi_{x_\phi}(t,\bar\xi)\| \\
      & \le 2 \|x_\phi(t,\xi) - x_\phi(t,\bar\xi)\| \\
      & \le 2 C \pfrac{\mu(t)}{\mu(s)}^a \nu(s)^\eps
               \|\xi - \bar\xi\|
   \end{align*}
   and the theorem is proved.
\end{proof}
\section{Behavior under perturbations}\label{section:perturbations}
In this section we assume that equation~\eqref{eq:ivp-li} admits a
$(\mu,\nu)$-dichotomy for some $D \ge 1$, $a <  0\le b$ and $\eps > 0$. Given
$c > 0$ and $q > 0$, let $\cP_{c,q}$ be the class of all perturbations $f
\colon [0, +\infty[ \times X \to X$ that verify conditions~\eqref{loc:cond-f-0}
and~\eqref{loc:cond-f-2} with the given $c$ and $q$. In $\cP_{c,q}$ we can
define a metric by
\begin{equation}\label{metric:P_c,q}
	\|f-\bar f\|'
   =\sup\set{\dfrac{\|f(t,u)-\bar f(t,u)\|}{\|u\|^{q+1}} \colon
      t \ge 0, u \in X \setm{0}},
\end{equation}
for every $f, \bar f \in \cP_{c,q}$.

The purpose of this section is to see how the manifolds in
Theorem~\ref{thm:local} vary with the perturbations. To do this we consider two
perturbation $f, \bar f \in \cP_{c,q}$ and the functions $\phi$ and $\bar\phi$
given by Theorem~\ref{thm:local} when we perturbe equation~\eqref{eq:ivp-li}
with $f$ and $\bar f$, respectively, and we compare the distance between $\phi$
and $\bar\phi$ in the metric given by~\eqref{metric:local:X} with the distance
between $f$ and $\bar f$ in the metric given by~\eqref{metric:P_c,q}.

\begin{theorem}
    Let $c > 0$ and $q > 0$. Suppose that equation~\eqref{eq:ivp-li} admits a
    $(\mu,\nu)$-dichotomy for some $D \ge 1$, $a <  0\le b$ and $\eps > 0$ and
    that the hypothesis of Theorem~\ref{thm:local} are satisfied. Then,
    choosing $\delta>0$ sufficiently small, there exists $K >0$ such that
		$$ \|\phi-\bar\phi\|' \le K \|f-\bar f\|'$$
   for every $f, \bar f \in \cP_{c,q}$, where $\phi, \bar \phi \in
   \cX_{\delta,\beta}$ are the functions given by Theorem~\ref{thm:local},
   corresponding to the perturbations $f$ and $\bar f$, respectively.
\end{theorem}

\begin{proof}
   Let $(s,\xi) \in G_{\delta,\beta}$. From~\eqref{loc:eq:phi} we obtain
	\begin{equation} \label{loc:eq:MajP3}
   	\|\phi(s,\xi)-\bar \phi(s,\xi)\|
      \le \int_s^{+\infty} \|V(r,s)^{-1}\| \cdot \|f_{x_\phi,\phi}(r,\xi)-\bar
            f_{x_{\bar \phi},\bar \phi}(r,\xi)\| \dr,
	\end{equation}
   where $x_\phi, x_{\bar\phi} \in \cB_{s,\delta,\beta}$ are the functions
   given by Lemma~\ref{lemma:local:aux1} associated with $(f,\phi)$ and $(\bar
   f, \bar \phi)$, respectively.
   By \eqref{metric:P_c,q}, \eqref{loc:cond-phi-2-2}, \eqref{loc:cond-f-2},
   \eqref{loc:cond-phi-2-3}, \eqref{loc:cond-x-2}
   and~\eqref{metric_B_s,delta,beta} we have for $r \ge s$
   \begin{equation} \label{norm_f_x,phi_-_bar-f_bar_x,bar-phi}
      \begin{split}
   		& \|f_{x_\phi,\phi}(r,\xi)
            - \bar f_{x_{\bar \phi},\bar \phi}(r,\xi)\|\\
   		& \le \|f_{x_\phi,\phi}(r,\xi)-\bar f_{x_\phi,\phi}(r,\xi)\|
            + \|\bar f_{x_\phi,\phi}(r,\xi) - \bar f_{x_{\bar \phi},
            \bar \phi}(r,\xi)\|\\
   		& \le 3^{q+1} \|f-\bar f\|'  \|x_\phi(r,\xi)\|^{q+1}
            + 3^{q+1}c  \|x_\phi(r,\xi)-x_{\bar\phi}(r,\xi)\|
            (\|x_\phi(r,\xi)\|+\|x_{\bar\phi}(r,\xi)\|)^q \\
   		& \quad + 3^q c \|\phi-\bar \phi\|' \|x_{\bar \phi}(r,\xi)\|
            \prts{\|x_\phi(r,\xi)\|+\|x_{\bar \phi}(r,\xi)\|}^q \\
   		& \le 3^{q+1} C^{q+1} \delta^q \|f-\bar f\|' \|\xi\|
            \pfrac{\mu(r)}{\mu(s)}^{aq+a}
            \nu(s)^{\eps(q+1)} \beta(s)^q\\
   		& \quad + 2^q 3^{q+1}c C^q \delta^q \|x_\phi-x_{\bar\phi}\|' \|\xi\|
            \pfrac{\mu(r)}{\mu(s)}^{aq+a} \nu(s)^{\eps (q+1)} \beta(s)^q\\
         & \quad + 2^q 3^q c C^{q+1} \delta^q \|\phi-\bar \phi\|' \|\xi\|
            \pfrac{\mu(r)}{\mu(s)}^{aq+a} \nu(s)^{\eps (q + 1)}
            \beta(s)^q
      \end{split}
   \end{equation}
   and using~\eqref{loc:equa:fund}, the last estimate,~\eqref{loc:eq:MajP3} and
   taking into account that $a-b < 0$, we get
   \begin{equation} \label{norm:phi(s,xi)-bar_phi(s,xi)}
      \begin{split}
         & \|\phi(s,\xi)-\bar \phi(s,\xi)\|\\
   		& \le 3^{q+1} C^{q+1} D \delta^q \|f-\bar f\|' \|\xi\|
            \mu(s)^{-aq} \nu(s)^{\eps (q + 1)} \beta(s)^q
            \int_s^{+\infty} \mu(r)^{aq} \nu(r)^\eps  \dr\\
   		& \quad + 2^q 3^{q+1}c C^{q} D \delta^q
            \|x_\phi-x_{\bar\phi}\|' \|\xi\|
            \mu(s)^{-aq} \nu(s)^{\eps (q+1)} \beta(s)^q
            \int_s^{+\infty} \mu(r)^{aq} \nu(r)^\eps  \dr\\
         & \quad + 2^q 3^q c C^{q+1} D \delta^q \|\phi-\bar \phi\|' \|\xi\|
            \mu(s)^{-aq} \nu(s)^{\eps (q + 1)} \beta(s)^q \int_s^{+\infty}
            \mu(r)^{aq} \nu(r)^\eps  \dr\\
         & \le 3^{q+1} C^{q+1} D \delta^q \|f-\bar f\|' \|\xi\|
            + 2^q 3^{q+1}c C^q D \delta^q \|x_\phi-x_{\bar\phi}\|' \|\xi\|\\
         & \quad + 2^q 3^q c C^{q+1} D \delta^q \|\phi-\bar \phi\|' \|\xi\|.
      \end{split}
   \end{equation}
   Now, we will estimate $\|x_\phi-x_{\bar \phi}\|'$. By~\eqref{eq:split-3a},
   ~\eqref{norm_f_x,phi_-_bar-f_bar_x,bar-phi} and~\eqref{eq:dich-1} we obtain
   for every $t \ge s$
	\begin{equation*}
		\begin{split}
         & \pfrac{\mu(t)}{\mu(s)}^{-a} \nu(s)^{-\eps}
            \|x_{\phi}(t,\xi)-x_{\bar\phi}(t,\xi)\| \\
   		& \le \pfrac{\mu(t)}{\mu(s)}^{-a} \nu(s)^{-\eps}
            \int_s^t \|U(t,r)\| \cdot \|f_{x_\phi,\phi}(r,\xi)
            - \bar f_{x_{\bar \phi},\bar \phi}(r,\xi)\| \dr \\
   		& \le 3^{q+1} C^{q+1} D \delta^q \|f-\bar f\|' \|\xi\|
            \mu(s)^{-aq} \nu(s)^{\eps q} \beta(s)^q
            \int_s^t \mu(r)^{aq} \nu(r)^\eps  \dr\\
   		& \quad + 2^q  3^{q+1}c C^{q} D \delta^q
            \|x_\phi-x_{\bar\phi}\|' \|\xi\|
            \mu(s)^{-aq} \nu(s)^{\eps q} \beta(s)^q
            \int_s^t \mu(r)^{aq} \nu(r)^\eps  \dr\\
         & \quad + 2^q 3^q c C^{q+1} D \delta^q \|\phi-\bar \phi\|' \|\xi\|
            \mu(s)^{-aq} \nu(s)^{\eps q} \beta(s)^q \int_s^t \mu(r)^{aq}
            \nu(r)^\eps  \dr\\
         & \le 3^{q+1} C^{q+1} D \delta^q \|f-\bar f\|' \|\xi\|
            + 2^q 3^{q+1}c C^q D \delta^q \|x_\phi-x_{\bar\phi}\|' \|\xi\|\\
         & \quad + 2^q 3^q c C^{q+1} D \delta^q \|\phi-\bar \phi\|' \|\xi\|
   	\end{split}
	\end{equation*}
   and using~\eqref{loc:equa:fund} this implies that
	\begin{equation*}
		\begin{split}
         & \|x_\phi - x_{\bar \phi}\|'\\
         & \le 3^{q+1} C^{q+1} D \delta^q \|f-\bar f\|'
            + 2^q 3^{q+1}c C^q D \delta^q \|x_\phi-x_{\bar\phi}\|'
            + 2^q 3^q c C^{q+1} D \delta^q \|\phi-\bar \phi\|'.
   	\end{split}
	\end{equation*}
   Thus, for $\delta>0$ such that $2^q 3^{q+1}c C^q D \delta^q < 1/2$ we have
      $$ \|x_\phi - x_{\bar \phi}\|'
         \le 2 \cdot 3^{q+1} C^{q+1} D \delta^q \|f-\bar f\|'
            + 2^{q+1} 3^q c C^{q+1} D \delta^q \|\phi-\bar \phi\|'.$$
   It follows from the last estimate,~\eqref{norm:phi(s,xi)-bar_phi(s,xi)} and
   $2^q 3^{q+1}c C^q D \delta^q < 1/2$ that
      $$ \|\phi(s,\xi) - \bar\phi(s,\xi)\|'
         \le 2 \cdot 3^{q+1} C^{q+1} D \delta^q \|f-\bar f\|' \|\xi\|
            + 2^{q+1} 3^q c C^{q+1} D \delta^q \|\phi-\bar\phi\|' \|\xi\|$$
   for every $(s, \xi) \in G_{\delta,\beta}$.
   Hence we get
      $$ \|\phi - \bar\phi\|'
         \le 2 \cdot 3^{q+1} C^{q+1} D \delta^q \|f-\bar f\|'
            + 2^{q+1} 3^q c C^{q+1} D \delta^q \|\phi-\bar\phi\|'$$
   and if we choose $\delta>0$ sufficiently small such that $2^{q+1} 3^q c
   C^{q+1} D \delta^q < 1/2$  we obtain
   	$$	\|\phi-\bar\phi\|'
         \le 4 \cdot 3^{q+1} C^{q+1} D \delta^q \|f-\bar f\|' $$
	and this proves the theorem.
\end{proof}

\section*{Acknowledgments}
This work was partially supported by FCT though Centro de Ma\-te\-m\'a\-ti\-ca da Universidade da Beira Interior (project PEst-OE/MAT/UI0212/2011).
\bibliographystyle{elsart-num-sort}

\end{document}